\documentclass[11pt]{article}
\usepackage{amssymb}
\usepackage{amsmath}
\input{epsf.sty}  
\newtheorem{proposition}{Proposition}
\newenvironment{proof}{\begin{trivlist}
\item[\hspace{\labelsep}{\bf\noindent Proof. }]}
{$\Box$\end{trivlist}}
\date{\empty}

\title{\textbf{\huge On a bilateral birth-death process\\
 with alternating rates\footnote{This paper is dedicated to the memory of Professor Luigi Maria Ricciardi, 
 who passed away in Naples on May 7, 2011.} 
 \footnote{Paper  published in Ricerche di Matematica, June 2012, Volume 61, Issue 1, pp 157-169.
 The final publication is available at http://link.springer.com/article/10.1007/s11587-011-0122-0.
 }}}

\author{\textbf{Antonio Di Crescenzo\footnote{Corresponding author; 
        tel.: +39-089-963349; fax: +39-089-963303}, Antonella Iuliano, Barbara Martinucci}\\
        Dipartimento di Matematica, Universit\`a di Salerno\\ 
        84084 Fisciano (SA), Italy\\       
        {\small Email: \{adicrescenzo, aiuliano, bmartinucci\}@unisa.it}
        }
\begin{document} 

\maketitle

\begin{abstract}
We consider a bilateral birth-death process characterized by a constant transition rate $\lambda$ 
from even states and a possibly different transition rate $\mu$ from odd states. We determine 
the probability generating functions of the even and odd states, the transition probabilities, 
mean and variance of the process for arbitrary initial state. 
Some features of the birth-death process confined to the non-negative integers by a reflecting 
boundary in the zero-state are also analyzed. In particular, making use of a Laplace 
transform approach we obtain a series form of the transition probability from state 1 to the 
zero-state. \\



\end{abstract}

%
\section{Introduction}
Birth-death processes were introduced to describe random growth (see, for instance, 
Ricciardi \cite{Ri86} for an accurate description of birth-death processes in the context of population 
dynamics). Furthermore, they arise as natural descriptors of time-varying phenomena in several 
applied fields such as queueing, epidemiology, epidemics, optics, neurophysiology, etc. 
An extensive survey has been provided in Parthasarathy and Lenin \cite{PaLe2004}. In particular, 
in Section 9 of that paper certain birth-death processes are used to describe the time changes 
in the concentrations of the components of a chemical reaction, and   their role  
in the study of diatomic molecular chains is emphasized. 
\par
Moreover,  Stockmayer et al.\ \cite{StGoNo1971} gave an 
example of application of stochastic processes in the study of chain molecular diffusion, 
by modeling a molecule as a freely-joined chain of two regularly alternating kinds of atoms. 
The two kinds of atoms have alternating jump rates, and these rates are reversed for odd labeled 
beads. By invoking the master equations for even and odd numbered bonds, the authors obtained 
the exact time-dependent average length of bond vectors. 
\par
Inspired by this work, Conolly et al.\ \cite{Coetal1997} studied an infinitely long chain of atoms 
joined by links of equal length. The links are assumed to be subject to random shocks, that force 
the atoms to move and the molecule to diffuse. The shock mechanism is different according to whether the 
atom occupies an odd or an even position on the chain. The originating stochastic model is a randomized 
random walk on the integers with an unusual exponential pattern for the inter-step time intervals. 
The authors analyze some features of this process and investigate also its queue counterpart, where 
the walk is confined to the non negative integers. Various results concerning this queueing system with 
``chemical'' rules (the so-called ``chemical queue'') were obtained also by Tarabia and 
El-Baz \cite{TaEl2002a}, \cite{TaEl2002b}  and more recently by Tarabia  et al.\ \cite{Taetal2009}. 
\par
Another example arising in a chemical context where the role of parity is crucial is provided in 
Lente \cite{Le2010}, 
where the probability of a more stable enantiomer is different according on whether the number 
of chiral molecules is even or odd. 
\par
Stimulated by the above researches, in this paper we consider a birth-death process $N(t)$ 
on the integers with a transition rate $\lambda$ from even states and a possibly different rate 
$\mu$ from odd states. This model arises by suitably modifying the death rates of the process 
considered in the above papers. A detailed description of the model is performed in Section 2, 
where the probability generating functions of even and odd states and the transition 
probabilities of the process are obtained for arbitrary initial state.  
Certain symmetry properties of the transition probabilities are also given. 
In Section 3 we study the birth-death process obtained by superimposing 
a reflecting boundary in the zero-state. In particular, by making use of a Laplace transform approach, 
we obtain the probability of a transition from state 1  to the zero-state. 
Formulas for mean and variance of both processes are also provided.
We remark that some preliminary results on the process under investigation 
are given in Iuliano and Martinucci \cite{IuMa2010} for the case of zero initial state. 
\par 
It should be mentioned that closed-form results on bilateral birth-death 
processes have been obtained in the past only in few solvable cases, such as those in the above 
mentioned papers, and those given in Di Crescenzo \cite{DiCr1994}, 
Di Crescenzo and Martinucci \cite{DiCrMa2009}, Pollett \cite{Po2001}.  
\section{Transient distribution}
%
%
\begin{figure}
\begin{center}
\begin{picture}(341,91) 
\put(39.25,45.25){\circle{38}} 
\put(99.25,45.25){\circle{38}} 
\put(159.25,45.25){\circle{38}} 
\put(219.25,45.25){\circle{38}} 
\put(279.25,45.25){\circle{38}} 
\put(9.25,64){\oval(45,10)[rt]} 
\put(69.25,64){\oval(45,10)[t]} 
\put(129.25,64){\oval(45,10)[t]} 
\put(189.25,64){\oval(45,10)[t]} 
\put(249.25,64){\oval(45,10)[t]} 
\put(309.25,64){\oval(45,10)[lt]} 
\put(9.25,26.3){\oval(45,10)[rb]} 
\put(69.25,26.3){\oval(45,10)[b]} 
\put(129.25,26.3){\oval(45,10)[b]} 
\put(189.25,26.3){\oval(45,10)[b]} 
\put(249.25,26.3){\oval(45,10)[b]} 
\put(309.25,26.3){\oval(45,10)[lb]} 
\put(19,35){\makebox(40,15)[t]{\Large $-2$}} 
\put(79,35){\makebox(40,15)[t]{\Large $-1$}} 
\put(140,35){\makebox(40,15)[t]{\Large $0$}} 
\put(200,35){\makebox(40,15)[t]{\Large $1$}} 
\put(260,35){\makebox(40,15)[t]{\Large $2$}} 
\put(5,69){\vector(1,0){10}} 
\put(65,69){\vector(1,0){10}} 
\put(125,69){\vector(1,0){10}} 
\put(185,69){\vector(1,0){10}} 
\put(245,69){\vector(1,0){10}} 
\put(305,69){\vector(1,0){10}} 
\put(15,21.3){\vector(-1,0){10}} 
\put(75,21.3){\vector(-1,0){10}} 
\put(135,21.3){\vector(-1,0){10}} 
\put(195,21.3){\vector(-1,0){10}} 
\put(255,21.3){\vector(-1,0){10}} 
\put(315,21.3){\vector(-1,0){10}} 
\put(35,70){\makebox(70,15)[t]{$\lambda$}} 
\put(105,70){\makebox(50,15)[t]{$\mu$}} 
\put(155,70){\makebox(70,15)[t]{$\lambda$}} 
\put(225,70){\makebox(50,15)[t]{$\mu$}} 
\put(45,-2){\makebox(50,15)[t]{$\mu$}} 
\put(95,-2){\makebox(70,15)[t]{$\lambda$}} 
\put(165,-2){\makebox(50,15)[t]{$\mu$}} 
\put(215,-2){\makebox(70,15)[t]{$\lambda$}} 
\put(3,69){\line(-1,0){1}} 
\put(3,21.3){\line(-1,0){1}} 
\put(0,69){\line(-1,0){1}} 
\put(0,21.3){\line(-1,0){1}} 
\put(-3,69){\line(-1,0){1}} 
\put(-3,21.3){\line(-1,0){1}} 
\put(317.5,69){\line(1,0){1}} 
\put(317.5,21.3){\line(1,0){1}} 
\put(320.5,69){\line(1,0){1}} 
\put(320.5,21.3){\line(1,0){1}} 
\put(323.5,69){\line(1,0){1}} 
\put(323.5,21.3){\line(1,0){1}} 
\end{picture} 
\end{center}
\vspace{-0.7cm}
\caption{Transition rate diagram of $N(t)$.}
\end{figure}
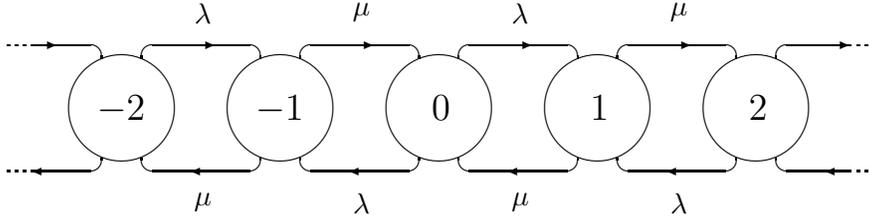
\par 
We consider a birth-death process $\{N(t);\,t\geq 0\}$ with state-space $\mathbb{Z}$, and denote by 
$$
p_{k,n}(t)=P\{{N}(t)={n}\,|\,{N}(0)={k}\},\quad t\geq 0,\quad n\in \mathbb{Z}
$$
its transition probabilities, where $k\in \mathbb{Z}$ is the initial state. 
We assume that $N(t)$ is characterized by a   
transition rate $\lambda$ from any even state to the two neighboring states, and by a 
possibly different transition rate $\mu$ from any odd state to the neighboring states. 
In other terms, denoting by 
$$
 \nu_{j,n}=\lim_{h\to 0}\frac{1}{h}\,P\{N(t+h)=n\,|\,N(t)=j\}
$$
the time-homogeneous transition rates of $N(t)$ from state $j$ to state $n$, we assume that  
the allowed transitions are characterized by the following rates:
\begin{equation}
\nu_{2n,2n\pm 1}=\lambda, 
\qquad 
\nu_{2n\pm 1,2n}=\mu,
\qquad \forall n\in \mathbb{Z},
\label{rates}
\end{equation}
with $\lambda,\mu >0$. 
The associated transition rate diagram of this process is given in Figure 1.
We note that rates (\ref{rates}) are different from those of the birth-death model 
considered in Conolly et al.\ \cite{Coetal1997} and Tarabia  et al.\ \cite{Taetal2009}, where 
$\nu_{2n,2n+ 1}=\nu_{2n+ 1,2n}=\lambda$ and $\nu_{2n- 1,2n}=\nu_{2n,2n- 1}=\mu$ 
for any $n\in \mathbb{Z}$. 
\par
Due to assumptions (\ref{rates}), the transition probabilities of $N(t)$ satisfy the following system of 
differential-difference equations:
\begin{equation}
\left\{ \begin{array}{l}
\displaystyle{\frac{{\rm d}}{{\rm d}t} p_{k,2n}(t)=\mu\, p_{k,2 n-1}(t)-2{\lambda}\, p_{k,2 n}(t)+\mu\, p_{k,2 n+1}(t)}, \\
\\
\displaystyle{\frac{{\rm d}}{{\rm d}t} p_{k,2 n+1}(t)=\lambda\, p_{k,2 n}(t)-2{\mu}\, p_{k,2 n+1}(t)+ \lambda\, p_{k,2  n+2}(t)},
\end{array} \right.
\label{eq1}
\end{equation}
for any $t\geq 0$, $n\in \mathbb{Z}$ and for any initial state $k\in \mathbb{Z}$. 
The initial condition is expressed by:
\begin{equation}
p_{k,n}(0)=\delta_{n,k},
\label{condizioneiniziale}
\end{equation}
where $\delta_{n,k}$ is the Kronecker's delta. We notice that in the special case when $\lambda=\mu$ 
process $N(t)$ identifies with the so-called ``randomized random walk'' (see Conolly \cite{Co1971}).
\par
In order to obtain the state probabilities of $N(t)$, hereafter we develop a probability 
generating function-based approach. We recall that this method has been used  
in the past to determine probabilities of interest in several stochastic models 
(see, for instance, Giorno and Nobile \cite{GiNo88} and  Ricciardi and Sato \cite{RiSa87} 
for the distribution of the range of one-dimensional random walks). 
Let us define the probability generating 
functions of the sets of even and odd states of $N(t)$, respectively:
\begin{eqnarray}
&& \hspace*{-1.2cm}
F_{k}(z,t):=\sum_{j=-\infty}^{+\infty} z^{2 j} p_{k,2 j}(t),
\qquad 
G_{k}(z,t):=\sum_{j=-\infty}^{+\infty} z^{2 j+1} p_{k,2 j+1}(t),
\label{defFG}
\end{eqnarray}
with $z \in \mathbb{Z}$. Note that, due to (\ref{condizioneiniziale}), the following initial conditions hold:
\begin{equation}
F_{k}(z,0)=\left\{ \begin{array}{ll}
\displaystyle{z^{k}} & {k\ \rm even} \\
\displaystyle{0} & {k\ \rm odd},
\end{array} \right.\ 
\qquad
G_{k}(z,0)=\left\{ \begin{array}{ll}
\displaystyle{0} & {k\ \rm even} \\
\displaystyle{z^{k}} & {k\ \rm odd}.
\end{array} \right.
\label{eq2}
\end{equation}
From system (\ref{eq1}) we have that the generating functions (\ref{defFG}) satisfy the following differential system:
$$
\left\{ \begin{array}{l}
\displaystyle{\frac{{\partial}}{{\partial} t} F_{k}(z,t)=\mu\, z G_{k}(z,t)-2{\lambda}\, F_{k}(z,t)+\frac{\mu}{z}\, G_{k}(z,t)}, \\
 \\
\displaystyle{\frac{{\partial}}{{\partial} t} G_{k}(z,t)=\lambda\, z F_{k}(z,t)-2{\mu}\, G_{k}(z,t)+\frac{\lambda}{z}\, F_{k}(z,t)},
\end{array} \right.
$$
so that 
\begin{displaymath}
\frac{{\partial}}{{\partial} t}
\left(\begin{array}{c}
F_{k}(z,t)  \\
G_{k}(z,t) 
\end{array} \right) =
A\cdot \left(\begin{array}{c}
F_{k}(z,t)  \\
G_{k}(z,t)  
\end{array} \right),\ \
\qquad
A :=
\left( \begin{array}{cc}
\displaystyle{-2 \lambda}  & \displaystyle{\mu \frac{z^2+1}{z}} \\
\displaystyle{\lambda \frac{z^2+1}{z}} & \displaystyle{-2 \mu} \\
\end{array} \right).
\end{displaymath}
Hence, by use of standard methods, due to conditions (\ref{eq2}) we come to 
\begin{equation}
\left(\begin{array}{c} 
\! F_{2k}(z,t)\! \\
\! G_{2k}(z,t) \!
\end{array} \right) 
= {\rm e}^{A t}
\cdot 
\left(\begin{array}{c}
z^{2k}  \\
0  
\end{array} \right),
\label{matriceA}
\end{equation}
and
\begin{equation}
\left(\begin{array}{c} 
\! F_{2k+1}(z,t)\! \\
\! G_{2k+1}(z,t) \!
\end{array} \right) 
= {\rm e}^{A t}
\cdot 
\left(\begin{array}{c}
0 \\
z^{2k+1} 
\end{array} \right),
\label{matriceAA}
\end{equation}
where
$$
{\rm e}^{A t}
= \exp \left\{ \left( \begin{array}{cc}
\displaystyle{-2 \lambda}  & \displaystyle{\mu \frac{z^2+1}{z}} \\
\displaystyle{\lambda \frac{z^2+1}{z}} & \displaystyle{-2 \mu} \\
\end{array} \right)\,t \right\}. 
$$
\par
\noindent
By straightforward calculations we have $A=S\cdot V\cdot S^{-1}$, where
\begin{equation}
S =
\left( \begin{array}{cc}
\displaystyle{\mu-\lambda-\frac{h(z)}{z}} & \displaystyle{\mu-\lambda+ \frac{h(z)}{z}}  \\
& \\
\displaystyle{\lambda \frac{z^2+1}{z}} & \displaystyle{\lambda \frac{z^2+1}{z}}  \\
\end{array} \right),
\qquad
V =
\left( \begin{array}{cc}
v_1 & 0  \\
& \\
0 & v_2 \\
\end{array} \right),
\label{matriceS}
\end{equation}
for $v_1= -(\lambda+\mu)-  h(z)/z$ and 
$v_2= -(\lambda+\mu)+  h(z)/z$, and where 
\begin{eqnarray}
&& \hspace*{-1.2cm}
S^{-1} =
-\frac{z}{2 \lambda (z^2+1) h(z)} 
\left( \begin{array}{cc}
\lambda (z^2+1) & z(\lambda-\mu)-h(z) \\
-\lambda (z^2+1) & z(\mu-\lambda)-h(z) 
\end{array} \right),
\label{matriceSinv}
\end{eqnarray}
with 
$$
 h(z):=\sqrt{(\mu z^2+\lambda) (\lambda z^2+\mu)}.
$$
If the initial state is even  ($2k$), Eqs.\ (\ref{matriceA}) and (\ref{matriceS})--(\ref{matriceSinv}) give
\begin{eqnarray*}
&& \hspace*{-0.8cm}
{\rm e}^{A t}
\cdot 
\left(\begin{array}{c}
z^{2k} \\
0  
\end{array} \right) 
=\left( \begin{array}{cc}
\displaystyle{\mu-\lambda- \frac{h(z)}{z}} & \displaystyle{\mu -\lambda+ \frac{h(z)}{z}} \\
& \\
\displaystyle{\lambda \frac{(z^2+1)}{z}} & \displaystyle{\lambda \frac{(z^2+1)}{z}}  \\
\end{array} \right)
\left( \begin{array}{cc}
\displaystyle{{\rm e}^{v_1 t}} & \displaystyle{0}  \\
& \\
\displaystyle{0} & \displaystyle{{\rm e}^{v_2 t}}  \\
\end{array} \right)
\left( \begin{array}{c}
\displaystyle{-\frac{z^{2k+1}}{2 h(z)}} \\
\\
\displaystyle{\frac{z^{2k+1}}{2 h(z)}}  \\
\end{array} \right),
\end{eqnarray*}
and then 
\begin{eqnarray}
&& \hspace*{-0.8cm}
{\rm e}^{A t}
\cdot 
\left(\begin{array}{c}
z^{2k} \\
0  
\end{array} \right) 
={\rm e}^{-(\lambda+\mu) t}\cdot{\frac{z^{2k}}{h(z)}}
\nonumber
\left( \begin{array}{c}
\displaystyle{{h(z)}\cosh[t\,\frac{h(z)}{z}]+z (\mu-\lambda) \sinh[t\,\frac{h(z)}{z}]}, \\
\\
\displaystyle{\lambda ({z^2+1}) \sinh[t\,\frac{h(z)}{z}]}  \\
\end{array} \right).
\nonumber
\\
&& \hspace*{0.8cm}
\label{FGsol}
\end{eqnarray}
Hence, from Eqs.\ (\ref{matriceA}) and (\ref{FGsol}) we obtain the explicit expression of the 
probability generating functions when the initial state is even:
\begin{eqnarray}
&& \hspace*{-1.5cm}
F_{2k}(z,t)={\rm e}^{-(\lambda+\mu) t}\frac{z^{2k}}{h(z)}
\left\{h(z) \cosh\left[\frac{t\,h(z)}{z}\right]
+z (\mu-\lambda) \sinh\left[\frac{t\,h(z)}{z}\right] \right\},
\label{Fsoluzione}
\\
\nonumber
\\
&& \hspace*{-1.5cm}
G_{2k}(z,t)={\rm e}^{-(\lambda+\mu) t}\frac{z^{2k}}{h(z)}
\lambda \left({z^2+1}\right) \sinh\left[\frac{t\,h(z)}{z}\right].
\label{Gsoluzione}
\end{eqnarray}
Similarly, if the initial state is odd  ($2k+1$) the explicit expression of the 
probability generating functions is:
\begin{eqnarray}
&& \hspace*{-1.cm}
F_{2k+1}(z,t)={\rm e}^{-(\lambda+\mu) t}\frac{z^{2k+1}}{h(z)}
\mu \left({z^2+1}\right) \sinh\left[\frac{t\,h(z)}{z}\right],
\label{FFsoluzione}
\\
\nonumber
\\
&& \hspace*{-1.cm}
G_{2k+1}(z,t)={\rm e}^{-(\lambda+\mu) t}\frac{z^{2k+1}}{h(z)}
\left\{h(z) \cosh\left[\frac{t\,h(z)}{z}\right]
+z (\lambda-\mu) \sinh\left[\frac{t\,h(z)}{z}\right] \right\}.
\nonumber
\\
\label{GGsoluzione}
\end{eqnarray}
We are now able to provide the state probabilities. 
\begin{proposition}\label{propProb}
For all $l,r \in \mathbb{Z}$ and $t \geq 0$ the transition probabilities of $N(t)$ are:
\begin{eqnarray}
&& \hspace*{-0.8cm}
p_{2l,2r}(t)={\rm e}^{-(\lambda+\mu) t} \sum_{n = {\left|r-l\right|}}^{+\infty}
\left[\frac{(\lambda t)^{2 n}}{(2 n)!}+\left(\frac{\mu-\lambda}{\lambda}\right) \frac{(\lambda t)^{2 n+1}}{(2 n+1)!} \right]
\nonumber
\\
&& \hspace*{0.6cm}
\times \sum_{k=0}^{n-{\left|r-l\right|}} {n \choose k} {n \choose k + {\left|r-l\right|}} \left(\frac{\lambda}{\mu}\right)^{-2 k - {\left|r-l\right|}},
\label{probparipari}
\\
&& \hspace*{-0.8cm}
\label{probpari1}
p_{2 l,2r+1}(t) = {\rm e}^{-(\lambda+\mu) t} \left\{ \sum_{n = {\left|r-l\right|}}^{+\infty} \frac{(\lambda t)^{2 n+1}}{(2 n+1)!}
\sum_{k=0}^{n-{\left|r-l\right|}} {n \choose k} {n \choose k+{\left|r-l\right|}} \left(\frac{\lambda}{\mu}\right)^{-2 k-{\left|r-l\right|}} \right.
\nonumber
\\
&& \hspace*{1cm}
\left. + \sum_{n={\left|r-l+1\right|}}^{+\infty} \frac{(\lambda t)^{2 n+1}}{(2 n+1)!} \sum_{k=0}^{n-{\left|r-l+1\right|}} {n \choose k} {n \choose k+{\left|r-l+1\right|}} \left(\frac{\lambda}{\mu}\right)^{-2 k-{\left|r-l+1\right|}}\right\}.
\nonumber
\\
&& \hspace*{0.8cm}
\label{probpari2}
\end{eqnarray}
\par
\noindent
\begin{eqnarray}
&& \hspace*{-0.8cm}
p_{2 l+1,2r}(t) = {\rm e}^{-(\lambda+\mu) t} \left\{ \sum_{n = {\left|r-l-1\right|}}^{+\infty} \frac{(\mu t)^{2 n+1}}{(2 n+1)!}
\sum_{k=0}^{n-{\left|r-l-1\right|}} {n \choose k} {n \choose k+{\left|r-l-1\right|}} \left(\frac{\mu}{\lambda}\right)^{-2 k-{\left|r-l-1\right|}} \right.
\nonumber
\\
&& \hspace*{1cm}
\left. + \sum_{n={\left|r-l\right|}}^{+\infty} \frac{(\mu t)^{2 n+1}}{(2 n+1)!} \sum_{k=0}^{n-{\left|r-l\right|}} {n \choose k} {n \choose k+{\left|r-l\right|}} \left(\frac{\mu}{\lambda}\right)^{-2 k-{\left|r-l\right|}}\right\}.
\\
&& \hspace*{-0.8cm}
\label{probdispari1}
p_{2l+1,2r+1}(t)={\rm e}^{-(\lambda+\mu) t} \sum_{n = {\left|r-l\right|}}^{+\infty}
\left[\frac{(\mu t)^{2 n}}{(2 n)!}+\left(\frac{\lambda-\mu}{\lambda}\right) \frac{(\mu t)^{2 n+1}}{(2 n+1)!} \right]
\nonumber  \\
&& \hspace*{1.5cm}
\times \sum_{k=0}^{n-{\left|r-l\right|}} {n \choose k} {n \choose k + {\left|r-l\right|}} \left(\frac{\mu}{\lambda}\right)^{-2 k - {\left|r-l\right|}}.
\label{probdispari2}
\end{eqnarray}
\label{proposition11}
\end{proposition}
\begin{proof}
It follows by extracting the coefficients of $x^{2 r}$ and $x^{2 r+1}$ in 
(\ref{Fsoluzione})--(\ref{GGsoluzione}), respectively.
\end{proof}
Figure 2 shows some plots of transition probabilities given in Proposition \ref{propProb}. 
%
\begin{figure}[t]  
\begin{center}
\epsfxsize=13.cm
\centerline{\epsfbox{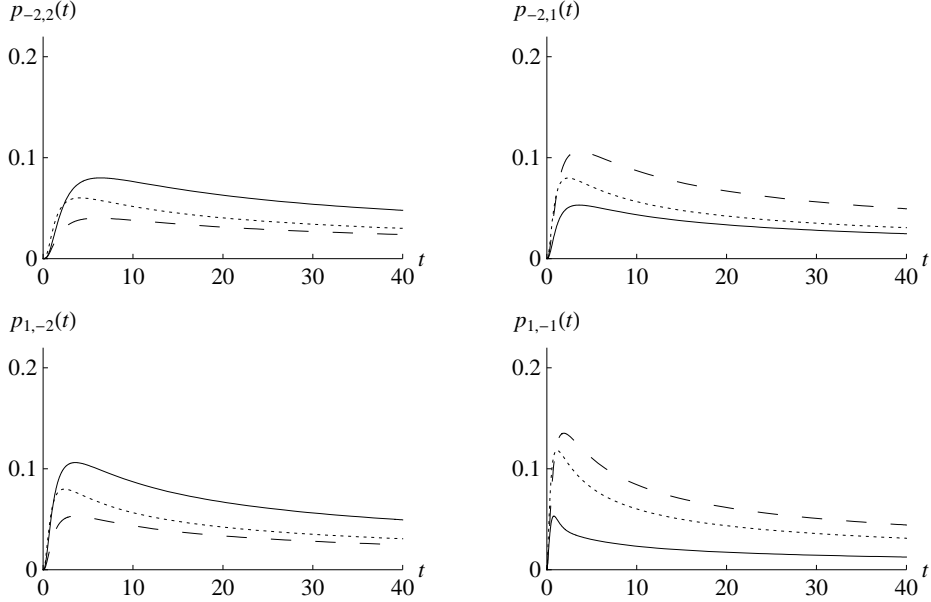}}
\caption{Plots of some transition probabilities for $(\lambda,\mu)=(1,2)$ (solid line), 
$(\lambda,\mu)=(2,2)$ (dotted line), $(\lambda,\mu)=(2,1)$ (dashed line).}
\end{center}
\end{figure}
%
\subsection{Symmetry properties}
%
The relevance of symmetry properties of transition functions of birth-death processes 
has been emphasized in Anderson and McDunnough \cite{AnMcDu1990} and 
in Di Crescenzo \cite{DiCr1998}, for instance. 
We stress that the role of symmetry is closely connected to the analysis of the 
first-passage-time problem in Markov process. See, for instance, the contributions 
of Giorno et al.\ \cite{Gietal89}, \cite{Gietal2010} 
and Di Crescenzo et al.\ \cite{DiCretal1995}, \cite{DiCretal1997}, 
where some relations involving the transition probability density functions and the 
first-passage-time density functions of symmetric diffusion processes 
in the presence of suitable time-varying boundaries. 
\par
Hereafter we analyze some symmetry properties of the transition probabilities obtained in Proposition 1. 
The proof is omitted, since it follows from direct analysis of the probabilities (\ref{probparipari})--(\ref{probdispari2}). 
When necessary we emphasize the dependence on the parameters by writing $p_{k,n}(t;\lambda,\mu)$ instead of $p_{k,n}(t)$.
\begin{proposition}\label{propsymm}
For every $t \geq 0$ and $n,k \in \mathbb{Z}$  the following symmetry relations hold:
$$
\begin{array}{lll}
(i) & p_{N-k,N-n}(t) = p_{k,n}(t), 
&  \hbox{if $N$ is  even} \\
(ii) & p_{N-k,N-n}(t; \lambda, \mu) = p_{k,n}(t;\mu, \lambda),\quad 
& \hbox{if $N$ is odd}; \\
(iii)\quad & p_{n,k}(t;\lambda,\mu) = p_{k,n}(t;\mu,\lambda);  
& {} \\
(iv) & p_{N+k,N+n}(t) = p_{k,n}(t), 
& \hbox{if $N$ is even}\\
(v) & p_{N+k,N+n}(t; \lambda, \mu) = p_{k,n}(t;\mu, \lambda),
& \hbox{if $N$ is odd}.
\end{array}
$$
\end{proposition}
\par
In Figure 2 the plots of $p_{-2,1}(t)$ and $p_{1,-2}(t)$    
illustrate a case in which property (ii) of Proposition \ref{propsymm} holds. 
\subsection{Moments}
Hereafter we obtain in closed form the  mean and the variance of $N(t)$. 
We shall obtain that the mean is equal to the initial state. This result is 
intuitively justified by the symmetry of the Markov chain. Indeed, by choosing 
$N=2k$ and $n=k-r$ in identity (i) of Proposition \ref{propsymm} we have 
$p_{k,k+r}(t)=p_{k,k-r}(t)$ $\forall k,r\in\mathbb{Z}$, and $t\geq 0$. 
\begin{proposition}
For $t \geq 0$ and $k \in \mathbb{Z}$ we have
\begin{eqnarray}
&& \hspace*{-1.1cm}
E[N(t)|N(0)=k]=k,
\label{media}
\\
&& \hspace*{-1.1cm}
Var[N(t)|N(0)=k]= \frac{4\lambda\mu }{\lambda+\mu}\,t+(-1)^k\frac{\lambda(\lambda-\mu)}{(\lambda+\mu)^2}
\left[1-{\rm e}^{-2(\lambda+\mu)t}\right],
\label{var}
\end{eqnarray}
\end{proposition}
%
\begin{proof}
The mean  (\ref{media}) easily follows from Eqs.\ (\ref{eq1}) and (\ref{condizioneiniziale}). 
Moreover, by setting $\psi_{k}(t):=E[N^2(t)|N(0)=k]$ from  system (\ref{eq1}) we obtain:
\begin{eqnarray*}
&& \hspace*{-0.8cm} 
\frac{{\rm d}}{{\rm d}t}\psi_{k}(t)= 2\mu\sum^{+\infty}_{n=-\infty}{p_{k,2n+1}}(t)
+2\lambda\sum^{+\infty}_{n=-\infty}p_{k,2n}(t)
\\
&& \hspace*{0.8cm}
=2\mu \,G_{k}(1,t)+2\lambda\,F_{k}(1,t), 
\qquad\qquad 
t \geq 0,
\end{eqnarray*}
where $F_{k}$ and $G_{k}$ have been defined in (\ref{defFG}). 
Hence, recalling Eqs.\ (\ref{Fsoluzione})--(\ref{GGsoluzione}), after some calculations we have 
$$
\frac{{\rm d}}{{\rm d}t}\psi_{k}(t)=\left\{ \begin{array}{l}
\displaystyle{
\frac{4\lambda\mu}{\lambda+\mu}+\frac{2\lambda(\lambda-\mu)}{\lambda+\mu}\, {\rm e}^{-2(\lambda+\mu)t}, \qquad k\ \rm even}
\\
\displaystyle{\frac{4\lambda\mu }{\lambda+\mu}+\frac{2\mu(\mu-\lambda)}{\lambda+\mu}{\rm e}^{-2(\lambda+\mu)t}, \qquad k\ \rm odd}
\end{array} \right.
$$
\nonumber
with $\psi_{k}(0)=k^{2}$. Finally, Eq.\ (\ref{var}) follows.
\end{proof}
\section{A reflecting boundary}
In this section we consider the case in which the state-space is reduced to 
the set of non-negative integers. We shall denote by $\{R(t);\,t\geq 0\}$ the birth-death process 
having state-space $\{0,1,2,\ldots\}$, with $0$ reflecting, whose rates are identical to 
those of $N(t)$. This describes, for instance,  the number of customers in a queueing system 
with alternating rates. For $n=0,1,2,\ldots$, let us introduce the transition probabilities
$$
q_{k,n}(t)=P\{{R}(t)={n}\,|\,{R}(0)={k}\},\quad t\geq 0.
$$
The related differential-difference equations are, for 
$n=1,2,\ldots,$

\begin{equation}
\left\{ \begin{array}{l}
\displaystyle{\frac{{\rm d}}{{\rm d}t} q_{k,0}(t)=\mu\, q_{k,1}(t)-{\lambda}\, q_{k,0}(t)},\\
 \\
\displaystyle{\frac{{\rm d}}{{\rm d}t} q_{k,2 n}(t)=\mu\, q_{k,2 n-1}(t)-2{\lambda}\, q_{k,2 n}(t)+\mu\, q_{k,2 n+1}(t)}, \\
 \\
\displaystyle{\frac{{\rm d}}{{\rm d}t} q_{k,2 n-1}(t)=\lambda\, q_{k,2 n}(t)-2{\mu}\, q_{k,2 n-1}(t)+ \lambda\, q_{k,2 n-2}(t)},
\end{array} \right.
\label{eq1new}
\end{equation}
with
\begin{equation}
q_{k,n}(0)=\delta_{n,k}.
\label{condizioneiniziale1}
\end{equation}
We point out that the steady-state distribution of $R(t)$ does not exist. 
Indeed, from system (\ref{eq1new}) it is not hard to see that 
$\displaystyle{\lim_{t\rightarrow\infty}q_{k,n}(t)=0}$ $\forall k,n \in \mathbb{Z}$.
\subsection{Moments}
Let us now set, for $k \in \mathbb{Z}$, 
\begin{equation}
P_{k}(t)=P\left\{R(t)\,{\rm even}\,|\,R(0)=k\right\}=\sum^{+\infty}_{n=0}{q_{k,2n}}(t),\qquad t \geq 0.
\label{pari}
\end{equation}
Now mean and variance of $R(t)$ will be formally expressed in terms of (\ref{pari}).
\begin{proposition}
For $t \geq 0$ we have
\begin{eqnarray}
&& \hspace*{-1.1cm}
\label{eqmean}
E[R(t)|R(0)=k]=\lambda\, {\int^{t}_{0}}q_{k,0}(\tau)d\tau + k,
\\
&& \hspace*{-1.1cm}
Var[R(t)|R(0)=k]=2(\lambda-\mu){\int^{t}_{0}}P_{k}(\tau)d\tau - \lambda(2k +1){\int^{t}_{0}}q_{k,0}(\tau)d\tau
\nonumber
\\
&& \hspace*{2.6cm}
-\lambda^2 \left[{\int^{t}_{0}}q_{k,0}(\tau)d\tau \right]^2+ 2\mu t,
\label{eqvar}
\end{eqnarray}
where  
\begin{eqnarray}
P_{k}(t)=\frac{2\mu}{\lambda+\mu}+\frac{\lambda-\mu}{\lambda+\mu}\, {\rm e}^{-2(\lambda+\mu)t}+ \lambda\,{\int^{t}_{0}}{\rm e}^{-2(\lambda+\mu)(t-\tau)}q_{k,0}(\tau)d\tau. 
\label{eqP}
\end{eqnarray}
\label{proposition}
\end{proposition}
\begin{proof}
The mean (\ref{eqmean}) easily follows from system (\ref{eq1new}) and condition (\ref{condizioneiniziale1}). 
Moreover, from Eqs. (\ref{eq1new}) we obtain 
\begin{eqnarray*}
&& \hspace*{-0.8cm}
\frac{{\rm d}}{{\rm d}t} E[R^2(t)|R(0)=k]= 2\mu\sum^{+\infty}_{n=1}{q_{k,2n+1}}(t)+2\mu q_{k,1}(t)+\lambda q_{k,0}(t)
+2\lambda{\sum^{+\infty}_{n=1}}q_{k,2n}(t)
\\
&& \hspace*{0.8cm}
=2\mu[1-P_{k}(t)-q_{k,1}(t)]+2\mu q_{k,1}(t)+\lambda q_{k,0}(t)+2\lambda\,[P_{k}(t)-q_{k,0}(t)]
\\
&& \hspace*{0.8cm}
=2(\lambda - \mu)P_{k}(t)+ \lambda q_{k,0}(t)+2 \mu,
\end{eqnarray*}
where $P_{k}(t)$ satisfies the differential equation
\begin{eqnarray}
\frac{{\rm d}}{{\rm d}t}P_{k}(t)= -2(\mu + \lambda)\, P_{k}(t)+ \lambda\,q_{k,0}(t)+2 \mu.
\label{A}
\end{eqnarray}
Since the solution of (\ref{A}) is Eq.\ (\ref{eqP}), the conditional variance (\ref{eqvar}) easily follows.
\end{proof}
%
\subsection{Probabilities}
%
We note that when $k=n=0$ the transition probability is given by 
(see Section 3 of Iuliano and Martinucci, 2010)
\begin{eqnarray*}
&& \hspace*{-0.7cm}
q_{0,0}(t)=\frac{{\rm e}^{-a t}}{a+b} \sum_{k=0}^{+\infty} \frac{(t/2)^{2k}}{k!^{2}}
\left\{ (a^{2 k+1}+b^{2 k+1})\,
{_{1}F_{2}}\left(-\frac{1}{2}, k+\frac{1}{2},k+1, \frac{b^2 t^2}{4}\right) \right.
\\
&& \hspace*{0.5cm}
\left. +\frac{t (a^{2 k+2}-b^{2 k+2})}{2(k+1)}\,
{_{1}F_{2}}\left(-\frac{1}{2}, k+1,k+\frac{3}{2}, \frac{b^2 t^2}{4}\right) \right\}, 
\qquad\qquad t\geq0,
\end{eqnarray*} 
where 
$$ 
a=\lambda+\mu,\qquad b=\lambda-\mu.
$$
Now we analyse the case in which the initial state is $k=1$.
Denoting by 
$$
\pi_{k,n}(s):={\cal L}_s [q_{k,n}(t)]=\int_{0}^{\infty} {\rm e}^{-s t}\,q_{k,n}(t)\,{\rm d} t,\quad s>0,
$$
the Laplace transform of the transition probabilities of $R(t)$, from Eqs. (\ref{eq1new}) we have:
\begin{equation}
\left\{ \begin{array}{l}
(\lambda +s)\, \pi_{1,0}(s)= \mu \pi_{1,1}(s) \\
(2 \mu+s)\, \pi_{1,1}(s)= 1+ \lambda \pi_{1,2}(s) + \lambda \pi_{1,0}(s) \\
(2 \lambda +s)\, \pi_{1,2 n}(s)= \mu \pi_{1,2 n-1}(s)+\mu \pi_{1,2 n+1}(s),\qquad n \geq 1 \\
(2 \mu +s)\, \pi_{1,2 n-1}(s)=\lambda \pi_{1,2 n}(s)+\lambda \pi_{1,2 n-2}(s),\qquad n\geq 2. \\
\end{array} \right.
\label{eq33}
\end{equation}
The solution of system (\ref{eq33}) involves the roots of the biquadratic equation 
$$
\lambda \mu\, x^4 -\big[(\lambda+\mu +s)^2-\lambda^2 -\mu^2\big]\,x^2 +\lambda \mu=0,
$$
which are  given by
\begin{eqnarray*}
&& \hspace*{-1.2cm}
\psi_1^2(s)=\frac{(A+B)^2}{a^2-b^2},\qquad \psi_2^2(s)=\frac{(A-B)^2}{a^2-b^2},
\label{equationSol}
\end{eqnarray*}
with 
$$
A^2=(a+s)^2-a^2,\qquad B^2=(a+s)^2-b^2.
$$ 
Since $\psi_1^2(s)>1$ and $0<\psi_2^2(s)<1$, from system (\ref{eq33})
we finally obtain: 
\begin{equation}
\pi_{1,2 n}(s)= 
\frac{(2\mu+s)(\lambda +s)\left[\psi_2^2(s)\right]^{n+1}}{\lambda^2 \left[\mu (1- \psi_2^2(s))- s\,\psi_2^2(s)\right]},\qquad n \geq 1,
\label{eqPP}
\end{equation}
and, similarly, 
\begin{eqnarray}
&& \hspace*{-0.6cm}
\pi_{1,2 n-1}(s)=
\frac{(\lambda +s)\left[\psi_2^2(s)\right]^n\left[1 + \psi_2^2(s)\right]}{\lambda 
\left[\mu (1- \psi_2^2(s))- s\,\psi_2^2(s)\right]},
\qquad n \geq 1.
\label{eqPP2}
\end{eqnarray}
By making use of Eqs.\ (\ref{eqPP}) and (\ref{eqPP2}) and substituting in (\ref{eq33}), we have
\begin{equation}
\pi_{1,0}(s)=\frac{(2 \lambda +s)(2\mu +s)-AB}{\lambda \,\left[s(2\mu+s)+AB\right]}.
\label{pi10}
\end{equation}
By inversion of (\ref{pi10}) after some calculations we obtain 
\begin{eqnarray}
&& \hspace*{-1.6cm}
q_{1,0}(t)=\frac{{\rm e}^{-a t}}{2 \lambda (a+b)}
\int_{0}^{t}\!\!\!\ \left[-b^2\,\frac{I_{1}(b(t-s))}{b(t-s)}+a^2\,\frac{I_{1}(a(t-s))}{a(t-s)}\right]h(s)\rm{d}s+
\nonumber
\\
&& \hspace*{-0.5cm}
+\frac{{\rm e}^{-a t}}{2 \lambda (a+b)}(a^2-b^2)
\int_{0}^{t}  \frac{b(t-s)}{2}\,{_{1}F_{2}}\left(\frac{1}{2}, \frac{3}{2},
2, \frac{b^2 (t-s)^2}{4}\right) h(s) \rm{d}s,
\label{int}
\end{eqnarray}
where 
$$
 h(x):= a\left[I_{0}(ax)+ I_{1}(ax)\right]+b\left[I_{0}(bx)- I_{1}(bx)\right],
$$ 
with $I_{n}(\cdot)$ denoting the modified Bessel function of the first kind. 
The evaluation of the integrals in Eq.\ (\ref{int}) finally gives the following result.
\begin{proposition}\label{propq10}
For $t\geq 0$, we have
\begin{eqnarray}
&& \hspace{-0.7cm}
q_{1,0}(t)=\frac{{\rm e}^{-a t}}{\lambda(a+b)}\bigg\{
\sum_{n=0}^{+\infty} \frac{t^{2n}}{n!(n+1)!}\, \left[\left(\frac{a}{2}\right)^{2 n+2}-\left(\frac{b}{2}\right)^{2 n+2}\right]\xi\left(\frac{1}{2},1,a,b\right)
\nonumber
\\
&& \hspace{0.4cm}
+\sum_{n=0}^{+\infty} \frac{t^{2n+1}}{n!(n+1)!(2n+1)}\,\left[\frac{a^2-b^2}{2}\left(\frac{b}{2}\right)^{2 n+1}\right]\xi\left(1,\frac{3}{2},a,b\right)
\nonumber
\\
&& \hspace{0.4cm}
+\sum_{n=0}^{+\infty} \frac{t^{2n+1}}{n!(n+1)!(2n+1)}\, \left[\left(\frac{a}{2}\right)^{2 n+2}-\left(\frac{b}{2}\right)^{2 n+2}\right]\eta\left(1,\frac{3}{2},a,b\right)
\nonumber
\\
&& \hspace{0.4cm}
+\sum_{n=0}^{+\infty} \frac{t^{2n+2}}{n!(n+1)!(2n+1)(2n+2)}\,\left[\frac{a^2-b^2}{2}\left(\frac{b}{2}\right)^{2 n+1}\right]\eta\left(\frac{3}{2},2,a,b\right)\bigg\},
\nonumber 
\end{eqnarray}
where
$$
\xi(u,v,a,b)={_{1}F_{2}}\left(\frac{1}{2}; n+u,n+v; \frac{a^2 t^2}{4}\right)-{_{1}F_{2}}\left(\frac{1}{2}; n+v,n+u; \frac{b^2 t^2}{4}\right),
$$
$$
\eta(u,v,a,b)=a\;{_{1}F_{2}}\left(\frac{1}{2}; n+u,n+v; \frac{a^2 t^2}{4}\right)+b\;{_{1}F_{2}}\left(\frac{1}{2}; n+v,n+u; \frac{b^2 t^2}{4}\right).
$$
\end{proposition}
\par
In conclusion, some illustrative plots of $q_{1,0}(t)$ are shown in Figure 3.
%
\begin{figure}[t]  
\begin{center}
\epsfxsize=7.5cm
\centerline{\epsfbox{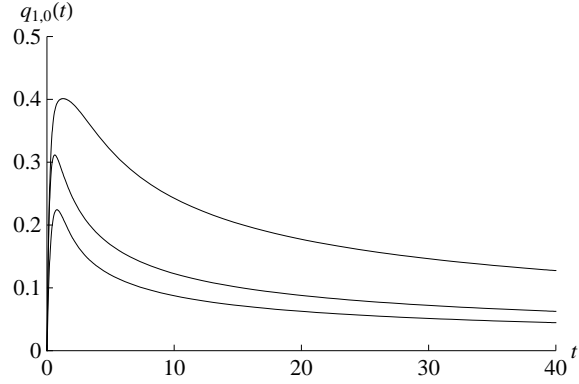}}
\caption{Plots of $q_{1,0}(t)$ for $(\lambda,\mu)=(1,2), (2,2), (2,1)$, from top to bottom.}
\end{center}
\end{figure}
%
\section{Concluding remarks}
Stimulated by some previous works on the applications of stochastic processes to 
the study of chain molecular diffusion, in this paper we have analyzed a birth-death 
process on ${\mathbb Z}$ characterized by alternating transition rates. 
The probability generating functions of even and odd states and the transition 
probabilities of the bilateral process have been obtained when the initial state is arbitrary.  
A preliminary investigation on the transient behavior of the birth-death process 
obtained by superimposing a reflecting boundary in the zero-state has also been performed. 
\par 
In conclusion, the results given in this paper deserve also special interest in the fields of 
chemical queueing processes and two-periodic random walks, according to the lines 
traced in various papers, such as Conolly et al.\ \cite{Coetal1997} 
and B\"ohm and Hornik \cite{BoHo2008}, for instance.
\subsection*{\bf Acknowledgments}
The authors thank Amelia G.\ Nobile for kind and helpful comments.
This work is partially supported by Ministero dell'Istruzione, dell'Universit\`a e della Ricerca 
(PRIN 2008).

\end{document}